\documentclass[a4paper,11pt]{article}
\usepackage{graphicx}
\usepackage{color}

\usepackage{amsmath,amsfonts,amssymb}   
\usepackage{latexsym}
\usepackage{amsthm}
\usepackage{a4wide}

\newtheorem{theorem}{Theorem}[section]
\newtheorem{proposition}[theorem]{Proposition}
\newtheorem{corollary}[theorem]{Corollary}

\newtheorem{remark}[theorem]{Remark}

\newcommand{\Mm}{\mathcal{M}^-_{\lambda, \Lambda}}
\newcommand{\Mp}{\mathcal{M}^+_{\lambda, \Lambda}}
\newcommand{\Mpm}{\mathcal{M}^\pm_{\lambda, \Lambda}}
\newcommand{\D}{\mathcal{D}^+}

\newcommand{\RN}{{\mathbb R}^N}
\newcommand{\R}{{\mathbb R}}
\newcommand{\xe}{x_\varepsilon}
\newcommand{\te}{t_\varepsilon}
\newcommand{\lte}{\lambda_{2,\varepsilon}}
\newcommand{\loe}{\lambda_{1,\varepsilon}}
\newcommand{\tuep}{\tilde u_\varepsilon^+}
\newcommand{\tuem}{\tilde u_\varepsilon^-}
\newcommand{\tuepm}{\tilde u_\varepsilon^\pm}
\newcommand{\pem}{p_\epsilon}
\newcommand{\pepm}{p^\pm_\epsilon}
\newcommand{\pep}{p^\epsilon}
\newcommand{\psp}{p^{\star}_{{}_+}}
\newcommand{\psm}{p^{\star}_{{}_-}}
\newcommand{\np}{\tilde N_{{}_{+}}}
\newcommand{\nm}{\tilde N_{{}_{-}}}
\newcommand{\npm}{\tilde N_{{}_{\pm}}}
\newcommand{\Um}{U^{{}^-}}
\newcommand{\Up}{U^{{}^+}}
\newcommand{\E}{E^{{}^\star}_{{}_+}}
\newcommand{\Em}{E^{{}^\star}_{{}_-}}
\newcommand{\Eeps}{E^{{}^+}_\varepsilon}
\newcommand{\Eems}{E^{{}^-}_\varepsilon}

\title{\textbf{Concentration and energy invariance for a class of fully nonlinear elliptic equations}}

\author{Isabeau Birindelli, Giulio Galise, Fabiana Leoni, Filomena Pacella \thanks{E-mail addresses: \texttt{isabeau@mat.uniroma1.it} (I. Birindelli), \texttt{galise@mat.uniroma1.it} (G. Galise), \texttt{leoni@mat.uniroma1.it} (F. Leoni), \texttt{pacella@mat.uniroma1.it} (F. Pacella)}}
 
\date{{\footnotesize Dipartimento di Matematica, Sapienza Universit\`a di Roma,  P.le Aldo Moro 2, I–00185 Roma, Italy}}

\begin{document}
\maketitle
\begin{abstract}
\noindent
We study the asymptotic behaviour of positive solutions of fully nonlinear elliptic equations in a ball, as the exponent of the power nonlinearity approaches a critical value. We show that solutions concentrate and blow up at the center of the ball, while a suitable associated energy remains invariant.
\end{abstract}

\vspace{1cm}

\noindent
\textbf{MSC 2010:} 35J60; 35B50; 34B15

\medskip
\noindent
\textbf{Keywords:} Fully nonlinear Dirichlet problems; asymptotic behaviour; critical exponents.

\bigskip 
\noindent
 Research partially supported by GNAMPA-INDAM and PRIN 2015 KB9WPT-PE1

\section{Introduction} The solvability of a general Lane-Emden equation of the form $-F(D^2u)=u^p$ strictly 
depends on the value of the exponent $p$.
Much has been done in the study of the behaviour of the solutions for $p$ near the threshold value that separates the 
existence from the non existence of solutions when $F(D^2u)=\Delta u$. In this paper,  we consider, instead, the fully 
nonlinear case when $F$ is either one of the extremal Pucci's operators and solutions are radial.

 Precisely we consider
\begin{equation}\label{eq1p}
\left\{
\begin{array}{cl}
-F(D^2u)=u^p & \text{in $B$}\\
u=0& \text{on $\partial B$,}
\end{array}\right.
\end{equation} 
where  $p>1$,  $B=B_1(0)$ is the unit ball in $\RN$ and $F$ is either the  maximal or the minimal Pucci's operator, i.e. 
$$
F(D^2u)=\Mp(D^2u)=\Lambda\sum_{\mu_i>0}\mu_i+\lambda\sum_{\mu_i<0}\mu_i
$$
or 
$$
F(D^2u)=\Mm(D^2u)=\lambda\sum_{\mu_i>0}\mu_i+\Lambda\sum_{\mu_i<0}\mu_i.
$$
Here $0<\lambda\leq\Lambda$ are the ellipticity  constants and $\mu_i=\mu_i(D^2u)$, $i=1,\ldots,N$, are the eigenvalues of the hessian matrix $D^2u$. \\
When we want to address both operators we will simply write $\Mpm$.

The existence or nonexistence of positive radial solutions of \eqref{eq1p} has been studied by Felmer and Quaas in \cite{FQ} (see also \cite{FQ2}). Their results show that there exists a critical exponent $\psp$ for $\Mp$ (resp. $\psm$ for $\Mm$) such that \eqref{eq1p} has a classical positive radial solution $u^+_p$ (resp. $u^-_p$ )  if, and only if, $p<\psp$ (resp. $p<\psm$). These critical exponents satisfy 
\begin{equation}\label{eq2p}
\max\left\{\frac{\np}{\np-2},\frac{N+2}{N-2}\right\}<\psp<\frac{\np+2}{\np-2}
\end{equation}
\begin{equation}\label{eq2}
\frac{\nm+2}{\nm-2}<\psm<\frac{N+2}{N-2}\,,
\end{equation}
where the dimension-like numbers $\tilde N_+$ and $\tilde N_-$ are defined by
$$
\np=\frac\lambda\Lambda(N-1)+1\qquad\text{and} \qquad \nm=\frac\Lambda\lambda(N-1)+1.
$$
We will always assume that $\npm>2$. Note that $p^{\star}_{{}_\pm}$ are the thresholds for the existence of solutions also for domains close to balls, as proved in \cite{EFQ}.\\
Since all positive solutions of \eqref{eq1p} are radial and radially decreasing by the symmetry results of \cite{DLS}, the previous statement apply to all positive solutions. Moreover it is easy to see, by the invariance of the equation under rescaling and the uniqueness theorem for the related ODE, that the solution is unique. \\
In the paper \cite{FQ} the critical exponent $\psp$ (resp. $\psm$) is defined by the property of being the only exponent for which there exists a \emph{fast decaying radial solution} of the analogous problem in $\RN$. This solution is unique, up to rescaling, and when $\lambda=\Lambda=1$, i.e. when the Pucci's operators reduce to the Laplacian, it is the well known function
\begin{equation}\label{talentiana}
V(x)=\frac{1}{\left(1+\frac{|x|^2}{N(N-2)}\right)^{\frac{N-2}{2}}}
\end{equation}
or one of its rescaling (see \cite{Au,Bl,Ta}). When $\lambda\neq\Lambda$ the fast decaying entire solutions  are not known explicitly and one of the results of the present paper is indeed to show that they are bounded from below and from above by functions similar to \eqref{talentiana}, see Section \ref{Sec2}, Theorem \ref{stime U}. 

Coming back to \eqref{eq1p}  we mention that existence results for similar Dirichlet problems in general bounded domains are contained in \cite{QS}, when the exponent $p$ is below $\frac{\npm}{\npm-2}$, see also \cite{CL}.
However these exponents are not sharp in the sense that they are not a boundary value between existence and nonexistence of positive solutions. We also recall that in the recent paper \cite{GLP} it is proved that radial positive solutions of the analogous problems in the annulus exist for all exponent $p>1$.

 The striking difference between the case of the ball and of the annulus suggests that the question of the existence or nonexistence of solutions of \eqref{eq1p} is quite delicate and may reflect some intrinsic  property of the Pucci's operators. Thus it is interesting to analyze the asymptotic behavior of positive solutions of  \eqref{eq1p} when the exponent $p$ approaches the critical exponents $p^{\star}_{{}_\pm}$ from below, in order to understand what happens in passing from the existence to the nonexistence range. This is indeed the aim of this paper.

 We consider subcritical exponents $p^+_\varepsilon=\psp-\varepsilon$ or $p^-_\varepsilon=\psm-\varepsilon$, for $\varepsilon>0$, and we denote simply by $u^\pm_\varepsilon$ the corresponding solutions of \eqref{eq1p}. The main result  of our paper shows that, as $\varepsilon\to0$, a concentration phenomenon appears,  more precisely the solutions  $u^\pm_\varepsilon$ blow up at the center of the ball having the profile of the fast decaying solutions of the analogous problems in $\RN$. Moreover, up to a multiplication of a suitable constant, they converge locally uniformly  in $B\backslash\left\{0\right\}$ to the \lq\lq Green functions\rq\rq\ of the Pucci's operators. Before stating our results let us observe that for any fixed $\varepsilon>0$ the functions  $u^\pm_\varepsilon$ achieve their maximum at the center of the ball, are radially decreasing and change convexity only once (see \cite{FQ}). Thus we set $r=|x|$, 
$$
M^\pm_\varepsilon:=\left\|u^\pm_\varepsilon\right\|_\infty=u^\pm_\varepsilon(0),
$$
and 
$r^\pm_0(\varepsilon)\in(0,1)$, the only radii such that 
\begin{equation}\label{r_0(p)}
\begin{split}
(u^\pm_\varepsilon)''(r)&<0\quad\text{for $r\in [0,r^\pm_0(\varepsilon))$},\\
(u^\pm_\varepsilon)''(r)&>0\quad\text{for $r\in(r^\pm_0(\varepsilon),1)$}.
\end{split}
\end{equation}
We have

\begin{theorem}\label{teorema1}
The following statements hold:
\begin{itemize}
	\item[i)] $\displaystyle\lim_{\varepsilon\to0}M^\pm_\varepsilon=+\infty$;
	\item[ii)] $u^\pm_\varepsilon\to0$  in $C^2_{\rm loc}(\overline{B}\backslash\left\{0\right\})$ as $\varepsilon\to0$;
	\item[iii)] the rescaled functions 
	$$
	\tilde u^\pm_\varepsilon(r)=\frac{1}{M^\pm_\varepsilon}u^\pm_\varepsilon\left(\frac{r}{{(M^\pm_\varepsilon)}^{\frac{p^\pm_\varepsilon-1}{2}}}\right),\qquad r=|x|<{(M^\pm_\varepsilon)}^{\frac{p^\pm_\varepsilon-1}{2}}
	$$
	converge in $C^2_{\rm loc}(\RN)$  to the fast decaying solutions of \eqref{eq3} or \eqref{eq3'}  as $\varepsilon\to0$;
\item[iv)] there exist positive constants $c^\pm_1$, depending only on $N, \Lambda$ and $\lambda$, such that, as $\varepsilon\to0$, 
$$
(M^\pm_\varepsilon)^{\frac{p^\pm_\varepsilon (\tilde{N}_\pm-2)-\tilde{N}_\pm}{2}}u^\pm_\varepsilon (r)\to c^\pm_1 \left( \frac{1}{r^{\tilde{N}_\pm-2}}-1\right) \quad \hbox{ in }\; C^2_{\rm loc} (\overline{B}\backslash\left\{0\right\}).
$$
\end{itemize}
\end{theorem}

When $\lambda=\Lambda=1$  the Pucci's operators $\Mpm$ reduce to the Laplace operator and then problem \eqref{eq1p}  becomes
\begin{equation}\label{eqL}
\left\{
\begin{array}{cl}
-\Delta u=u^p & \text{in $B$}\\
u=0& \text{on $\partial B$.}
\end{array}\right.
\end{equation}
It is well know that solutions of \eqref{eqL} exist if and only if $p$ is less than the critical exponent $p^\star=\frac{N+2}{N-2}=2^\star-1$, where $2^\star$ is the Sobolev exponent for the embedding of $H^1_0(\Omega)$ into $L^q(\Omega)$.\\
Let us point out that $p^\star$ gives a sharp bound for the existence of solutions,  not only in the ball but also in any star-shaped domain, in view of the famous Pohozaev identity \cite{P}. \\
The nonexistence of solutions of \eqref{eqL} for $p\geq p^\star$, also in more general bounded domains $\Omega$, is strictly related to le lack of compactness for the embedding $H^1_0(\Omega)\hookrightarrow L^{2^\star}(\Omega)$ which does not allow to use the standard variational methods to find critical points of the associated energy functional:
\begin{equation}\label{energyfunctional}
J(u)=\frac12\int_{\Omega}\left|Du\right|^2\,dx-\frac{1}{p^\star+1}\int_{\Omega}u^{p^\star+1}\,dx.
\end{equation}
Indeed the Palais-Smale compactness condition for $J$ fails at the energy levels $\frac kN S^N$, for $k\in\mathbb N$, where $\displaystyle S=\inf_{u\in H^1_0(\Omega)\backslash\left\{0\right\}}\frac{\left\|Du\right\|_{L^2(\Omega)}}{\left\|u\right\|_{L^{2^\star}(\Omega)}}$ is the best Sobolev constant for the above critical embedding. The failure of this compactness condition has been proved to be related to concentration phenomena as shown in several papers, starting from \cite{BC1,BC2},  \cite{Lions1,Lions2}, \cite{Str}, not only for \eqref{eqL} but for many other variational problems. In the particular case when the domain is a ball then the unique positive radial solution $u_\varepsilon$ of \eqref{eqL} for $p=p_\varepsilon=p^\star-\varepsilon$ blows up and concentrates at the center of the ball as $\varepsilon\to0$, while for the energy it holds:
\begin{equation}\label{energyfunctional2}
J(u_\varepsilon)=\left(\frac12-\frac{1}{p_\varepsilon+1}\right)\int_B u_\varepsilon^{p_\varepsilon+1}\,dx\;\stackrel{\varepsilon\to0}{\longrightarrow}\;\frac1N S^{N}.
\end{equation}
Moreover the local profile of $u_\varepsilon$ is that of a positive solution of the analogous problem in the whole $\RN$.\\
Such a behavior of the solutions $u_\varepsilon$, as $\varepsilon\to0$, depends also on the invariance by rescaling of both the equation \eqref{eqL} and the energy \eqref{energyfunctional} when the exponent is $p^\star$.
Indeed it is important to observe that the equation in \eqref{eqL} is invariant by the rescaling 
$$
v_\alpha(x)=\alpha v\left(\alpha^{\frac{p-1}{2}}x\right)\qquad\text{for $\alpha>0$}
$$
\emph{for any exponent $p>1$}, but the critical exponent $p^\star$ \emph{is the only exponent} for which also the energy $\int_\Omega u^{p^\star+1}\,dx$ is invariant by the same rescaling.

Therefore, coming back to the fully nonlinear problem \eqref{eq1p}, we can say that Theorem \ref{teorema1} extends the classical concentration results holding for the problem \eqref{eqL}. In our opinion this could not have been easily predictable by what is known for the Pucci's operators, since neither a lack of compactness nor a clear energy invariance by rescaling is related to the definitions of $\psp$ and $\psm$. Indeed problem \eqref{eq1p} does  not have a variational structure and the critical exponents $\psp$ and $\psm$ are only characterized by the fact that they are the unique exponents for which a fast decaying radial solution of the analogous problems  in the whole $\RN$ exists.\\
This reflects in the difficulty in proving Theorem \ref{teorema1} since we cannot exploit energy bounds as 
in \cite{H} to prove the results.\\
Instead we use the related ODE and prove some estimates which allow to detect the behavior of the solutions as $\varepsilon\to0$. We point out that, while for the operator $\Mm$ the estimates can be proved by considering 
Pohozaev type functionals,  as in \cite{AP},  in the case of the operator $\Mp$   new arguments relying on a phase plane analysis are needed. Let us emphasize that we cannot use directly the stable-unstable manifold theorem (see \cite{Hale}), we provide instead ad hoc analysis of the trajectories.

Even though our problem does not have a variational structure, we show that it is possible to define a weighted energy $\E(u)$ in the space of radial functions changing convexity only once which is invariant by rescaling and is finite for the fast decaying entire solutions $U^{{}^\pm}$ defined in Section \ref{Sec2}. Moreover $\E(U^{{}^+})$ represents the limits of the weighted energies of the solutions $u^+_\varepsilon$, hence obtaining  a formula analogous to \eqref{energyfunctional2}. Indeed we have
\begin{theorem}\label{4th1}
As $\varepsilon\to0$ then
\begin{equation}\label{4eq9}
\int_{B}{(u_\varepsilon^+)}^{p^+_\varepsilon+1}g^+_{u^+_\varepsilon}\,dx\to\int_{\RN}({\Up})^{\psp+1}g_{\Up}^+\,dx=\Sigma<+\infty
\end{equation}
where $g_{\Up}^+$ and $g^+_{u^+_\varepsilon}$ are defined in \eqref{4eq2} and \eqref{4eq8}. Analogous statement holds for $u^-_\varepsilon$ with obvious changes.
\end{theorem}
This energy analysis completes the study of the concentration behavior of the solutions $u^\pm_\varepsilon$, as $\varepsilon\to0$, since it  shows that while they blow up at the origin and converge to zero in $B\backslash\left\{0\right\}$ they keep some \lq\lq mass\rq\rq .
We refer to Section \ref{Sec4} for more comments.

A final remark is that there are other classes of fully nonlinear operators for which there is a clear invariance by rescaling of the related equations and of some associated energy functionals. These are the so called $k$-Hessian operators which have a variational structure and for which related critical exponents can be defined, sharing many similarities with the case of the Laplacian (\cite{Lei}, \cite{Tso}).

\medskip
The paper is organized as follows: in Section \ref{Sec2} we give global estimates for the fast decaying radial solutions $U^{{}^\pm}$; moreover  an integral characterization of the critical exponents $p^\star_\pm$ is provided; Section \ref{Sec3} is devoted to the proof of Theorem \ref{teorema1}; in Section \ref{Sec4} we analyze various aspects and invariance properties of the weighted energies and prove Theorem \ref{4th1}.

\section{The critical equation in $\RN$}\label{Sec2}
We first recall some facts about the critical character of $\psp$ and $\psm$. Consider the problem
\begin{equation}\label{eq3}
\left\{
\begin{array}{cl}
-\Mp(D^2u)=u^p & \text{in $\RN$}\\
u\geq0 & \text{in $\RN$}.\\
\end{array}\right.
\end{equation}
The following statements have been proved in \cite{FQ}:
\begin{itemize}
	\item[i)] if $p<\psp$ then the only  radial solution of \eqref{eq3} is $u\equiv0$;
	\item[ii)] if $p=\psp$ then there is a unique radial solution $\Up$ satisfying $\Up(0)=1$; moreover it is \emph{fast decaying}, i.e. 
	\begin{equation}\label{Ustar}
	\lim_{r\to+\infty}r^{\np-2}\,\Up(r)=c^+_1,\qquad \lim_{r\to+\infty}r^{\np-1}\,(\Up)'(r)=-(\tilde{N}_+-2) c_1^+
	\end{equation}
	for a positive constant $c^+_1$.\end{itemize}
	Analogous results hold in the case of the Pucci's minimal operator $\Mm$, for the problem
	\begin{equation}\label{eq3'}
\left\{
\begin{array}{cl}
-\Mm(D^2u)=u^p & \text{in $\RN$}\\
u\geq0 & \text{in $\RN$},\\
\end{array}\right.
\end{equation}
simply replacing $\psp$ by $\psm$ and $\np$ by $\nm$. In particular, the unique radial solution  of \eqref{eq3'} at the critical level $p=\psm$, say $\Um=\Um(|x|)$, such that $\Um(0)=1$, satisfies, for a positive constant $c^-_1$  
	\begin{equation}\label{Ustar'}
	\lim_{r\to+\infty}r^{\nm-2}\,\Um(r)=c^-_1,\qquad \lim_{r\to+\infty}r^{\nm-1}\,(\Um)'(r)=- (\tilde{N}_{-}-2) c_1^-.
	\end{equation}
Both \eqref{Ustar} and \eqref{Ustar'} are obtained, after performing the Lane Emden transformation, as an application 
of the stable-unstable manifold theorem in the 
phase plane analysis. In Section \ref{Sec3}, we shall give a 
characterization of $c^\pm_1$ that will be used  to prove \emph{iv)} of Theorem \ref{teorema1}.
 
Note that the fast decaying functions $U^\pm$ have the same monotonicity and convexity properties of $u_\varepsilon^\pm$, in particular
\begin{equation}\label{eq3''}
\left\|U^\pm\right\|_\infty=U^\pm(0)=1.
\end{equation}
Denoting by $R^\pm_0$ the unique radius such that $(U^\pm_1)''(R^\pm_0)=0$, it is easy to see that these functions satisfy the following ODE: 
\begin{equation}\label{eqU+}
\begin{cases}
\displaystyle(\Up)''(r)+(N-1)\frac{(\Up)'(r)}{r}=-\frac{({\Up})^{\psp}(r)}{\lambda}&\text{for $r\in[0,R^+_0]$}\medskip\\
\displaystyle(\Up)''(r)+(\np-1)\frac{(\Up)'(r)}{r}=-\frac{({\Up})^{\psp}(r)}{\Lambda}&\text{for $r\in[R^+_0,+\infty)$}
\end{cases}
\end{equation}
and
\begin{equation}\label{eqU-}
\begin{cases}
\displaystyle(\Um)''(r)+(N-1)\frac{(\Um)'(r)}{r}=-\frac{({\Um})^{\psm}(r)}{\Lambda}&\text{for $r\in[0,R^-_0]$}\medskip\\
\displaystyle(\Um)''(r)+(\nm-1)\frac{(\Um)'(r)}{r}=-\frac{({\Um})^{\psm}(r)}{\lambda}&\text{for $r\in[R^-_0,+\infty)$}.
\end{cases}
\end{equation}
For any $\alpha,\beta\in\R$ we consider the following Pohozaev type functionals:
\begin{equation}\label{H}
\begin{split}
H^{^{-}}_{\alpha,\beta}(r)&=r^{\nm}\left([(\Um)']^2+\frac{\alpha}{\psm+1}({\Um})^{\psm+1}\right)+\beta r^{\nm-1}({\Um})'\Um\\
H^{^{+}}_{\alpha,\beta}(r)&=r^{\np}\left([(\Up)']^2+\frac{\alpha}{\psp+1}({\Up})^{\psp+1}\right)+\beta r^{\np-1}({\Up})'\Up\,.
\end{split}
\end{equation}
Using the equations \eqref{eqU+}-\eqref{eqU-} one has
\begin{equation}\label{derH}
\begin{split}
(H^{^{-}}_{\alpha,\beta})'(r)&=(2+\beta-\nm)r^{\nm-1}[(\Um)']^2+\left(\frac{\alpha\nm}{\psm+1}-\frac\beta\lambda\right)r^{\nm-1}({\Um})^{\psm+1}\\
&\quad+\left(\alpha-\frac2\lambda\right)r^{\nm}({\Um})^{\psm} (\Um)'\qquad\text{for $r\in[R^-_0,+\infty)$}\\
(H^{^{+}}_{\alpha,\beta})'(r)&=(2+\beta-\np)r^{\np-1}[(\Up)']^2+\left(\frac{\alpha\np}{\psp+1}-\frac\beta\Lambda\right)r^{\np-1}({\Up})^{\psp+1}\\
&\quad+\left(\alpha-\frac2\Lambda\right)r^{\np}({\Up})^{\psp} (\Up)'\qquad\text{for $r\in[R^+_0,+\infty)$.}
\end{split}
\end{equation}

\begin{theorem}\label{stime U}
There exist positive constants $c^\pm<C^\pm$ such that for any $r\geq R^+_0$ and any $r\geq R^-_0$ one has respectively
\begin{equation}\label{stima U+}
\frac{\Up(R^+_0)}{\left(1+C^+(r^2-(R^+_0)^2)\right)^{\frac{\np-2}{2}}}\leq\Up(r)\leq\frac{\Up(R^+_0)}{\left(1+c^+(r^2-(R^+_0)^2)\right)^{\frac{\np-2}{2}}}
\end{equation}
and
\begin{equation}\label{stima U-}
\frac{\Um(R^+_0)}{\left(1+C^-(r^2-(R^-_0)^2)\right)^{\frac{\nm-2}{2}}}\leq\Um(r)\leq\frac{\Um(R^-_0)}{\left(1+c^-(r^2-(R^-_0)^2)\right)^{\frac{\nm-2}{2}}}\;.
\end{equation}
\end{theorem}
\begin{proof}
We detail the proof of \eqref{stima U+}, the other one being similar.  To simplify the notations we avoid sub and superscripts: $H,U,\tilde N,p^{\star}$ stand for $H^{^{+}}_{\alpha,\beta}, \Up, \np, \psp$ respectively.
If we fix in \eqref{derH}
$$
\alpha=\frac{(\tilde N-2)(p^{\star}+1)}{\Lambda\tilde N}\qquad{\text{and}}\qquad\beta=\np-2
$$
 we obtain
$$
H'(r)=\frac{p^{\star}(\tilde N-2)-(\tilde N+2)}{\Lambda\tilde N}r^{\tilde N}U^{p^{\star}} U'>0,
$$
since $U$ is decreasing and $p^{\star}<\frac{\tilde N+2}{\tilde N-2}$. Moreover $\lim_{r\to+\infty}H(r)=0$, hence $H$ is negative for any $r\geq R^+_0$. On the other hand by a straightforward computation
$$
\left(U^{-\frac{\tilde N}{\tilde N-2}}\frac{U'}{r}\right)'=-\frac{\tilde N}{\tilde N-2}U^{\frac{2(1-\tilde N)}{\tilde N-2}}r^{-(\tilde N+1)}H(r)\geq0.
$$
We deduce that $U^{-\frac{\tilde N}{\tilde N-2}}\frac{U'}{r}$ is increasing and for $r\geq R^+_0$
$$
U^{-\frac{\tilde N}{\tilde N-2}}(R^+_0)\frac{U'(R^+_0)}{R^+_0}\leq U^{-\frac{\tilde N}{\tilde N-2}}\frac{U'}{r}\leq \lim_{r\to+\infty}U^{-\frac{\tilde N}{\tilde N-2}}\frac{U'}{r}=-(\tilde N-2)c_1^{-\frac{2}{\tilde N-2}},
$$
where $c_1$ is the positive constant appearing in \eqref{Ustar}. Integrating the above inequalities
$$
U^{-\frac{\tilde N}{\tilde N-2}}(R^+_0)\frac{U'(R^+_0)}{R^+_0}\int_{R^+_0}^r s\,ds\leq\int_{R^+_0}^r U^{-\frac{\tilde N}{\tilde N-2}}U'\,ds\leq -c_1(2-\tilde N){c_1^{\frac{\tilde N}{2-\tilde N}}}\int_{R^+_0}^r s\,ds.
$$
Then the conclusion follows by taking $$C^+=\frac{U^{p^{\star}-1}(R^+_0)}{\Lambda(\tilde N-1)(\tilde N-2)}\qquad\text{and}\qquad c^+=c_1\left(c_1^{-\tilde N}U^2(R_0^+)\right)^{\frac{1}{\tilde N-2}}.
$$
If the operator $\Mm$ is considered the constants $C^-$ and $c^-$ in \eqref{stima U-}  are 
$$
C^-=c_1\left(c_1^{-\tilde N}U^2(R_0^-)\right)^{\frac{1}{\tilde N-2}}\qquad\text{and}\qquad c_-=\frac{U^{p^{\star}-1}(R^-_0)}{\lambda(\tilde N-1)(\tilde N-2)}
$$
where now $U,\tilde N,p^{\star}$ are equal to $\Um, \nm, \psm$ respectively and 
$c_1$ is the analogous constants of \eqref{Ustar} in the case of $\Mm$.
\end{proof}

Using the identities
\begin{equation}\label{intH'}
\int_{R^\pm_0}^{+\infty}(H^{^{\pm}}_{\alpha,\beta})'(r)\,dr=-H^{^{\pm}}_{\alpha,\beta}(R^\pm_0)
\end{equation}
we obtain the following 

\begin{proposition}
Let $K_0^+=\left[\Up(R^+_0)\right]^{\psp-1}(R^+_0)^2$. Then the critical exponent $\psp$ satisfies the equation
$$
(\np+2-\psp(\np-2))\int_{R^+_0}^{+\infty}r^{\np-1}({\Up})^{\psp+1}=\frac{(\Up(R^+_0))^{\psp+1} (R^+_0)^{\np}}{\np-1}\left[1-\frac{\psp+1}{\Lambda(\np-1)}K_0^+\right]
$$
Similar identity (with obvious changes) holds for $\psm$. 
\begin{proof}
Choose $\alpha=\frac2\Lambda$ and $\beta=\np-2$ in \eqref{derH}. In this case
$$
\int_{R^\pm_0}^{+\infty}(H^{^{+}}_{\alpha,\beta})'dr=\frac{\np+2-\psp(\np-2)}{\Lambda(\psp+1)}\int_{R^+_0}^{+\infty}r^{\np-1}({\Up})^{\psp+1}dr
$$
and from \eqref{H}
\begin{equation*}
\begin{split}
(H^{^+}_{\alpha,\beta})(R^+_0)&=(R^+_0)^{\np}\left([(\Up)']^2(R^+_0)
+\frac{2}{\Lambda(\psp+1)}({\Up})^{\psp+1}(R^+_0)\right)\\&\quad+(\np-2) r^{\np-1}({\Up})'(R^+_0)\Up(R^+_0).
\end{split}
\end{equation*}
Since $({\Up})'(R^+_0)=-\frac{R^+_0}{\Lambda(\np-1)}({\Up(R^+_0)})^{\psp}$ we  obtain the thesis by a straightforward computation. 
\end{proof}

\end{proposition}

\section{Asymptotic behavior of the solutions}\label{Sec3}
We start by proving statement $i)$ of Theorem \ref{teorema1}.
\begin{proposition}\label{propMepsilon} 
 $M^\pm_\varepsilon=u^\pm_{\varepsilon}(0)=\|u^\pm_{\varepsilon}\|_\infty$ satisfies
\begin{equation}\label{Mepsilon}
\lim_{\varepsilon\to0}M^\pm_\varepsilon=+\infty
\end{equation}
\end{proposition}
\begin{proof}
We detail the proof in the case of $M^+_\varepsilon$. If \eqref{Mepsilon} does not hold, then for a sequence $\varepsilon_n\to0$ we have that $$\lim_{n\to\infty}M^+_{\varepsilon_n}=C.$$
If $C>0$ then, by regularity estimates the sequence of solutions $u^+_{\varepsilon_n}$ would converge in $C^2(\overline B)$ to a nontrivial solution  of \eqref{eq1p} at the critical level $p=\psp$, which does not exist by \cite{FQ}. Let us exclude now that $C=0$. For this let us consider the principal eigenvalue $\lambda^+_1=\lambda^+_1(-\Mp,B)$ defined by
$$
\lambda^+_1=\sup\left\{\lambda:\;\exists \varphi\in C(B),\;\varphi>0,\;\;-\Mp(D^2\varphi)\geq\lambda\varphi\;\;\text{in $B$}\right\}.
$$
The number $\lambda^+_1$ is well defined, positive and it gives a threshold for the validity of the maximum principle, i.e. for any $\lambda<\lambda^+_1$ and any solution $u$ of
\begin{equation}\label{maxprin}
-\Mp(D^2u)\leq \lambda u\;\;\text{in $B$},\quad u\leq0\;\;\text{on $\partial B$}
\end{equation}
then necessarily $u\leq0$ in $B$ (see \cite{BEQ, BD}). \\
If for a sequence $M^+_{\varepsilon_n}\to0$ as $n\to+\infty$, we can pick $n$ large enough such that ${(M^+_{\varepsilon_n})}^{p_{\varepsilon_n}-1}<\lambda^+_1$. Then
$$
\left\{
\begin{array}{cl}
-\Mp(D^2u^+_{\varepsilon_n})\leq {(M^+_{\varepsilon_n})}^{p_{\varepsilon_n}-1}u^+_{\varepsilon_n} & \text{in $B$}\\
u^+_{\varepsilon_n}>0 & \text{in $B$}\\
u^+_{\varepsilon_n}=0 & \text{on $\partial B$},
\end{array}\right.
$$
 a contradiction with \eqref{maxprin}.
\end{proof}

Since $M^\pm_\varepsilon\to+\infty$ as $\varepsilon\to0$ we will use them as rescaling parameters to study the asymptotic behavior of $u^\pm_\varepsilon$ as  $\varepsilon\to0$. Let us consider the rescaled functions
\begin{equation}\label{utildeepsilon}
\tilde u^\pm_\varepsilon(x)=\frac{1}{M^\pm_\varepsilon}u^\pm_\varepsilon\left(\frac{x}{{(M^\pm_\varepsilon)}^{\frac{p^\pm_\varepsilon-1}{2}}}\right)\qquad x\in\tilde B^\pm_\varepsilon={(M^\pm_\varepsilon)}^{\frac{p^\pm_\varepsilon-1}{2}}B
\end{equation}
which solve
\begin{equation}\label{eq4}
\left\{
\begin{array}{cl}
-{\mathcal M}_{\lambda,\Lambda}^\pm(D^2\tilde u^\pm_\varepsilon)=\tilde u_\varepsilon^{p_\varepsilon} & \text{in $\tilde B^\pm_\varepsilon$}\\
\tilde u^\pm_\varepsilon>0 & \text{in $\tilde B^\pm_\varepsilon$}\\
\tilde u^\pm_\varepsilon=0& \text{on $\partial\tilde B^\pm_\varepsilon$}
\end{array}\right.
\end{equation} 
and
\begin{equation}\label{eq5}
\tilde u^\pm_\varepsilon(0)=\left\|\tilde u^\pm_\varepsilon\right\|_\infty=1.
\end{equation}

\medskip
\begin{proposition}\label{propfundamental}
The following statements hold:
\begin{itemize}
	\item[i)] $\tilde u^\pm_\varepsilon$ converge  to $U^{{^\pm}}$ in $C^2_{{\rm loc}}(\RN)$ as  $\varepsilon\to0$;
	\item[ii)] $\displaystyle \lim_{\varepsilon\to0}{\left[r^\pm_0(\varepsilon)\right]}^{\frac{2}{p^\pm_\varepsilon-1}}{M^\pm_\varepsilon}={\left(R^\pm_0\right)}^{\frac{2}{p_\pm^\star-1}}$;
	\item[iii)] $\displaystyle \lim_{\varepsilon\to0}\frac{u^\pm_{\varepsilon}(r_0(\varepsilon))}{M^\pm_\varepsilon}=U^{{^\pm}}(R^\pm_0)$;
	\item[iv)] $\displaystyle\lim_{\varepsilon\to0}{\left[r^\pm_0(\varepsilon)\right]}^{\frac{2}{p^\pm_\varepsilon-1}}u^\pm_{\varepsilon}(r_0(\varepsilon))={\left(R^\pm_0\right)}^{\frac{2}{p_\pm^\star-1}}U^{{^\pm}}(R^\pm_0)$.
\end{itemize}
\end{proposition}
\begin{proof}
By Proposition \ref{propMepsilon}  $M^\pm_\varepsilon\to+\infty$ as $\varepsilon\to0$, hence the limit of the domains $\tilde B^\pm_\varepsilon$ is the whole $\RN$. Since the families ${(\tilde u^\pm_\varepsilon)}_\varepsilon$ are bounded in the $L^\infty$ norm, by elliptic estimates, they converge, as  $\varepsilon\to0$ in  $C^2_{{\rm loc}}(\RN)$, respectively to  radial solutions of extremal problems in the whole $\RN$ having maximum at $0$ equal to $1$, i.e. to $\Up$ and $\Um$.

\noindent
For \emph{ii)} let us consider the radii
$$
\tilde r^\pm_0(\varepsilon)={(M^\pm_\varepsilon)}^{\frac{p^\pm_\varepsilon-1}{2}}r^\pm_0(\varepsilon).
$$
By the very definition of the rescaled functions \eqref{utildeepsilon} we have
$$
(\tilde u^\pm_\varepsilon)''(\tilde r^\pm_0(\varepsilon))=0
$$
and 
$$
(\tilde u^\pm_\varepsilon)''(r)<0\quad\text{for $r\in[0,\tilde r^\pm_0(\varepsilon))$}.
$$
Then $\tilde r^\pm_0(\varepsilon)$ cannot converge to $0$ otherwise by the $C^2_{{\rm loc}}$ convergence of $\tilde u^\pm_\varepsilon$ to $U^{{}^\pm}$  we would have that $(U^{{}^\pm})''(0)=0$ but this is in contradiction with the equations \eqref{eqU+}-\eqref{eqU-} satisfied by $U^{{}^\pm}$, which in particular yield  
$$
(\Um)''(0)=-\frac{1}{\Lambda N}\quad\text{and}\quad(\Up)''(0)=-\frac{1}{\lambda N}.
$$
Moreover $\tilde r^\pm_0(\varepsilon)$ cannot blow up to $+\infty$, otherwise for  $R>R^\pm_0$ we would have $(\tilde u^\pm_\varepsilon)''(r)<0$ in $[0,R]$ if $\varepsilon$ is small enough. Using again the local convergence of $\tilde u^\pm_\varepsilon$ to $U^{^{\pm}}$ we then would have $(U^{{^\pm}})''(R)<0$, while $(U^{{^\pm}})''(R)>0$ since $U^{{^\pm}}$ change convexity only once, exactly to $R^\pm_0$.\\
Hence $\tilde r^\pm_0(\varepsilon)$ are bounded and they must necessarily converge to $R^\pm_0$ respectively, because these are the only radii where $(U^{{^\pm}})''$ vanish. \\
So we have got \emph{ii)} which immediately implies \emph{iii)} since
$$
u^\pm_\varepsilon(r^\pm_0(\varepsilon))=M^\pm_\varepsilon\tilde u^\pm_\varepsilon(\tilde r^\pm_0(\varepsilon))
$$
and, by \emph{i)}, $\tilde u^\pm_{\varepsilon}(\tilde r^\pm_0(\varepsilon))\to U^{{}^\pm}(R^\pm_0)$ as $\varepsilon\to0$. Statement \emph{iv)} is a consequence of \emph{ii)}-\emph{iii)}.
\end{proof}

\medskip
\begin{remark}
\rm From Proposition \ref{propfundamental} we deduce:
\begin{equation}\label{invariance}
{\left[r^\pm_0(\varepsilon)\right]}^{\frac{2}{p^\pm_\varepsilon-1}}u^\pm_\varepsilon(r_0^\pm(\varepsilon))=
{\left[\tilde r^\pm_0(\varepsilon)\right]}^{\frac{2}{p^\pm_\varepsilon-1}}\tilde u^\pm_\varepsilon(\tilde r_0^\pm(\varepsilon)).
\end{equation}
\end{remark}

\medskip
\begin{corollary}\label{estib}
There exist two positive  constants $C_{{}_+}$, $C_{{}_-}$ such that 
\begin{equation}\label{stime2}
(\tuep)^\prime(r)\leq -C_{{}_+} r^{1-\np}\quad\mbox{and } \quad (\tuem)^\prime(r)\leq -C_{{}_-} r^{1-\nm}.
\end{equation}
\end{corollary}
\begin{proof}
Observe that from the equation satisfied by $\tuep$, for $r\in\left[\tilde r_0^+(\varepsilon),(M^+_\varepsilon)^\frac{p^+_\varepsilon-1}{2}\right]$
$$(r^{\np-1}(\tuep)^\prime(r))^\prime<0\,.$$
Hence
\begin{equation*}
\begin{aligned}
(\tuep)^\prime(r)r^{\np-1}&\leq (\tuep)^\prime(\tilde r_0^+(\varepsilon))\tilde r_0^+(\varepsilon)^{\np-1}\\
&=-\frac{1}{(\np-1)\Lambda} (\tuep)^{p_\varepsilon}(\tilde r_0^+(\varepsilon))\tilde r_0^+(\varepsilon)^{\np-1}
\end{aligned}
\end{equation*}
In view of Proposition \ref{propfundamental} the right hand side is convergent and this proves the first estimate of \eqref{stime2}. The other one is similar.
\end{proof}


\bigskip

In the following proposition we obtain global  universal bounds for 
$$
\frac{\tilde u^\pm_\varepsilon(r)}{r^{2-\tilde N_{_{\pm}}}}\qquad\text{and}\qquad\frac{\left|{(\tilde u^\pm_\varepsilon)}'(r)\right|}{r^{1-\tilde N_{_{\pm}}}}.
$$
\begin{proposition}\label{esti}
There exists a  positive   constant  $C=C(\lambda,\Lambda,N)$  such that
\begin{equation}\label{phsp}
\frac{\tuepm(r)}{r^{2-\npm}}\leq C \qquad\text{and}\qquad\frac{|(\tuepm)^\prime(r)|}{r^{1-\npm}}\leq C .\end{equation}
Moreover, for any $r^\pm_\varepsilon\to +\infty$ as $\varepsilon\to 0$, with $r^\pm_\varepsilon\leq (M^\pm_\varepsilon)^{\frac{\pepm-1}{2}}$ and
$$
\frac{(M^\pm_\varepsilon)^{\frac{\pepm-1}{2}}}{r^\pm_\varepsilon}\to l_\pm\in [1,+\infty] ,
$$
one has
\begin{equation}\label{conv}
\frac{\tuepm(r^\pm_\varepsilon)}{(r^\pm_\varepsilon)^{2-\npm}}\to c^\pm_1 \left( 1- l_\pm^{2-\npm}\right)\quad \hbox{ as } \varepsilon\to 0\, ,
\end{equation}
where $c^\pm_1$ are the positive constants appearing in \eqref{Ustar} and \eqref{Ustar'}.
\end{proposition}
\begin{proof} The proof is done by using the Emden-Fowler transformation that reduces the equations satisfied by 
$\tuem$ and $\tuep$ to autonomous ODEs. We shall give the proof of \eqref{phsp} and \eqref{conv} only for $\tuep$, the interested reader will easily see how to adjust it in order to
obtain the results  for $\tuem$.

We divide the proof into two steps. In the first one, after some general considerations, we prove preliminary facts.

\medskip
\noindent \textbf{Step 1.} For any function $f$ that is homogenous of degree $p>1$, let $\lambda_2=\frac{2}{p-1}$ and suppose that,
for some $A\in\R$,
$u=u(r)$ is a solution for $r>0$ of the equation
\begin{equation}\label{uu} u^{\prime\prime}+A\frac{u^\prime}{r}=f(u).
\end{equation}
Then $$x(t)=r^{\lambda_2}u(r),\quad r=e^t$$  is a solution for  $t\in\R$ of
\begin{equation}\label{xx}
x^{\prime\prime}-(\lambda_1+\lambda_2)x^\prime+\lambda_1\lambda_2 x=f(x)
\end{equation}
where $\lambda_1=\lambda_2-A+1$. The solutions of \eqref{xx} satisfy
\begin{equation}\label{xx1}
\begin{split}
x(t)=&x_{-}e^{\lambda_1(t-T)}+x_{+}e^{\lambda_2(t-T)}\\
 &\quad + \frac{e^{\lambda_2 t}}{\lambda_2-\lambda_1}\int_T^tf(x(s))e^{-\lambda_2 s}ds-\frac{e^{\lambda_1 t}}{\lambda_2-\lambda_1}\int_T^tf(x(s))e^{-\lambda_1 s}ds,
 \end{split}
\end{equation}
and
\begin{equation}\label{xprime}
\begin{split}
x^\prime(t)=&\lambda_1\left( x_{-}e^{\lambda_1(t-T)}-\frac{e^{\lambda_1 t}}{\lambda_2-\lambda_1}\int_T^tf(x(s))e^{-\lambda_1 s}ds\right)\\
&\quad \quad + \lambda_2\left(x_{+}e^{\lambda_2(t-T)}+\frac{e^{\lambda_2 t}}{\lambda_2-\lambda_1}\int_T^tf(x(s))e^{-\lambda_2 s}ds\right)\\
\end{split}
\end{equation}
with $T\in \R$ fixed and $x(T)=x_{-}+x_{+}$, $x^\prime(T)=\lambda_1x_{-}+\lambda_2x_{+}$.

Recall that for $\varepsilon>0$, $p^+_\varepsilon=\psp-\varepsilon$ and  $\tuep$ is a solution of \eqref{uu} with $A=\np-1$ and $f(u)=-\frac{u^{p^+_\varepsilon}}{\Lambda}$ for $r\geq \tilde r_0^+(\varepsilon)$. Analogously, the function $U^+$ introduced in Section \ref{Sec2} is a solution of \eqref{uu} with $A=\np-1$ and $f(u)=-\frac{u^{\psp}}{\Lambda}$ for $r\geq R_0^+$.  Let
$$
\lambda_1^\star=-\frac{\psp(\np-2)-\np}{\psp-1},\quad\lambda_2^\star=\frac{2}{p^\star_+-1}
$$
and
$$
\loe=-\frac{p^+_\varepsilon(\np-2)-\np}{p_\varepsilon-1},\quad\lte=\frac{2}{p^+_\varepsilon-1}.
$$
Note that, by \eqref{eq2p}, $\lambda_1^\star$ is negative, as well as $\loe$ for sufficiently small $\varepsilon$, whereas $\lambda_2^\star, \lte>0$.\\
Let
$$
 X(t)=r^{\lambda_2^\star}\Up(r)\quad \text{and}\quad \xe(t)=r^{\lte}\tuep(r),\quad r=e^t,
$$
 and let $T_\varepsilon= \frac{(p_\varepsilon^+-1)}{2}\log M^+_\varepsilon$ indicate the value such that $\xe(T_\varepsilon)=0$; observe that $\xe^\prime(T_\varepsilon)<0$. When dealing with the function $\xe(t)$, we shall always suppose that $t<T_\varepsilon$.\\
Using the properties of $\tuep$ and $\Up$, the following facts hold:

\begin{enumerate}
\item[1)] $\displaystyle\lim_{t\rightarrow+\infty}(X(t),X^\prime(t))=(0,0)$ and $\displaystyle\lim_{t\rightarrow+\infty}e^{-\lambda_1^\star t} (X(t),X^\prime(t))=\left(c_1^+,  c_1^+\lambda_1^\star\right)$.
\item[2)] $(\xe,\xe^\prime)$ converges   to $(X,X^\prime)$ as $\varepsilon\to 0$, uniformly in $(-\infty,\tau]$ for any fixed $\tau$.
\item[3)] There exists $t_{1,\varepsilon}$  such that $\xe^\prime( t_{1,\varepsilon})=0$ and $\xe^\prime(t)<0$ for any $t>t_{1,\varepsilon}$. Similarly,  there exists $T_{1,\star}$  such that $X^\prime( T_{1,\star})=0$ and $X^\prime(t)<0$ for any $t>T_{1,\star}$.
\end{enumerate}
Statement 1) is actually the characterization of $p_+^\star$ (see \cite{FQ}), and it is nothing but a reformulation  of \eqref{Ustar}. Point 2) is a consequence of Proposition \ref{propfundamental}. We shall give the proof of 3).
The phase plane associated to the equation \eqref{xx} satisfied by $\xe$ has two critical points 
$(0,0)$ and  $((|\loe|\lte\Lambda)^{\frac{1}{\pep-1}},0)$. Furthermore the 
trajectories that go from the first quadrant to the fourth quadrant need to 
cross the $x$-axis at $x> (|\loe|\lte\Lambda)^{\frac{1}{\pep-1}}$, because of the sign of
$-\loe\lte x+f(x)$.

The trajectory of $\xe$ starts in the first quadrant, so it will cross the $x$-axis in order to reach $(\xe(T_\varepsilon),\xe^\prime(T_\varepsilon))=(0, \xe^\prime(T_\varepsilon))$ which is at the boundary of the fourth quadrant. 
Hence the existence of $t_{1,\varepsilon}$ such that $\xe^\prime(t_{1,\varepsilon})=0$. Observe that $\xe^\prime(t)<0$ for $t\in (t_{1,\varepsilon}, T_\varepsilon)$. Indeed  if it crossed the $x$ axis again it would need to do so at a value where $\xe<(|\loe|\lte\Lambda)^{\frac{1}{\pep-1}}$, but then, since the trajectories cannot
self cross, $\xe(t)$ would never reach zero, see the figure below.
\begin{center}
		\includegraphics[width=0.50\textwidth]{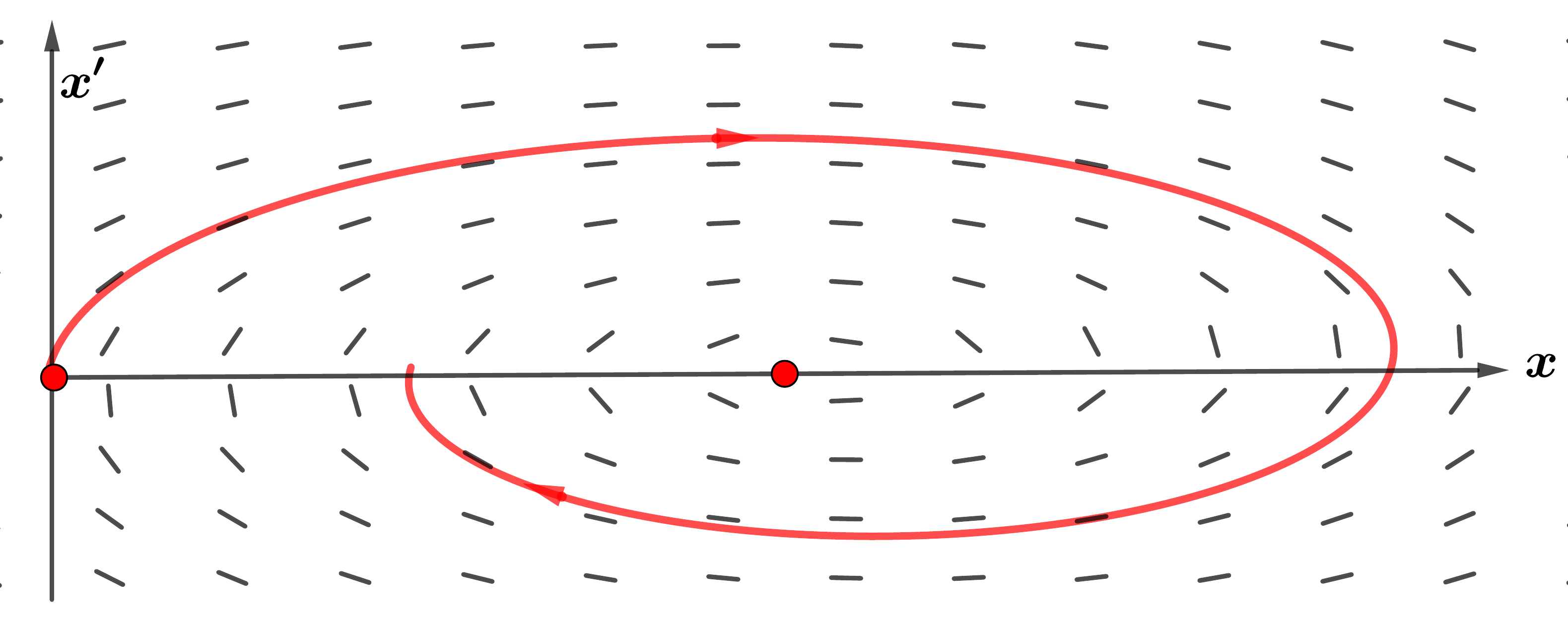}
\end{center}

\medskip
 
\noindent
\textbf{Step 2.} Fix $\varepsilon_0>0$ small. 
Let $\delta>0$ be a universal value such that for all $\varepsilon<\varepsilon_0$
\begin{equation}\label{delta}
\frac{\delta^{p_\varepsilon-1}}{\lte(\lte-\loe)\Lambda}<\frac{1}{2}\quad\text{and}\quad 
\left(\loe+\frac{2\delta^{p_\varepsilon-1}}{(\lte-\loe)\Lambda}\right)\pem<\loe.
\end{equation}

We choose $T$ such that $X(T)<\delta$ and $X^\prime(T)<0$. Reducing $\varepsilon_0$ if necessary, we may assume as well that for any $\varepsilon<\varepsilon_0$
$$\xe(T)<\delta\ \mbox{and}\  \xe^\prime(T)<0;$$
hence, using point 3) above, $\xe(t)<\delta$ for any $t>T$. Without loss of generality we shall also suppose that $T>\max \{ \log R_0^+, \log(\tilde r_0^+(\varepsilon))\}$, since by Proposition \ref{propfundamental} the value $\tilde r_0^+(\varepsilon)$ is bounded for $\varepsilon$ small. Note also that $T>\max \{ T_{1,\star}, t_{1,\varepsilon}\}$.

Let us use the representation formula \eqref{xx1} both for $X$ and for $\xe$, namely 
$$
\begin{array}{rl}
X(t)= & X_{-}e^{\lambda^\star_1(t-T)}+X_{+}e^{\lambda^\star_2(t-T)} \\[2ex]
& \displaystyle +\frac{e^{\lambda^\star_1 t}}{\Lambda (\lambda^\star_2-\lambda^\star_1)}\int_T^t X^{\psp}(s) e^{-\lambda^\star_1 s}ds
- \frac{e^{\lambda^\star_2 t}}{\Lambda( \lambda^\star_2-\lambda^\star_1)}\int_T^t X^{\psp}(s)e^{-\lambda^\star_2 s}ds
\end{array}
$$
and
$$
\begin{array}{rl}
\xe(t)= & x_{\varepsilon,{}_-}e^{\loe (t-T)}+ x_{\varepsilon,{}_+}e^{\lte (t-T)} \\[2ex]
& \displaystyle +\frac{e^{\loe t}}{\Lambda (\loe-\lte)}\int_T^t \xe^{p^+_\varepsilon}(s) e^{-\loe s}ds
- \frac{e^{\lte t}}{\Lambda( \loe-\lte)}\int_T^t \xe^{p^+_\varepsilon}(s)e^{-\lte s}ds,
\end{array}
$$
with $x_{\varepsilon,{}_-}\to X_{-}$ and $x_{\varepsilon,{}_+}\to X_+$ as $\varepsilon\to 0$ by point 2) above.

Since $\lambda_1^\star<0$ and $\lambda_2^\star>0$, point 1) of the previous step implies
$$
X(t)= X_{-}e^{\lambda^\star_1(t-T)} 
+\frac{e^{\lambda^\star_1 t}}{\Lambda (\lambda^\star_2-\lambda^\star_1)}\int_T^t X^{\psp}(s) e^{-\lambda^\star_1 s}ds
+\frac{e^{\lambda^\star_2 t}}{\Lambda( \lambda^\star_2-\lambda^\star_1)}\int_t^{+\infty} X^{\psp}(s)e^{-\lambda^\star_2 s}ds
$$
 and, moreover, 
 \begin{equation}\label{relc1}
c_1^+=\lim_{t\to +\infty} e^{-\lambda_1^\star t}X(t)= X_{-}e^{-\lambda^\star_1T} 
+\frac{1}{\Lambda (\lambda^\star_2-\lambda^\star_1)}\int_T^{+\infty}X^{\psp}(s) e^{-\lambda^\star_1 s}ds\, .
\end{equation}
As far as $\xe$ is concerned,  from $\xe(T_\varepsilon)=0$  we deduce  
\begin{equation}\label{xx2}
\begin{split}
\xe(t)&=x_{\varepsilon,{}_-} e^{\loe(t-T)} +\frac{e^{\loe t}}{(\lte-\loe)\Lambda}\int_T^t \xe^{p^+_\varepsilon}(s)e^{-\loe s}ds\\
&\quad-x_{\varepsilon,{}_-} e^{\loe(T_\varepsilon-T)}e^{\lte(t-T_\varepsilon)} -\frac{e^{\loe T_\varepsilon}e^{\lte(t-T_\varepsilon)}}{(\lte-\loe)\Lambda}\int_T^{T_\varepsilon} \xe^{p^+_\varepsilon}(s)e^{-\loe s}ds\\
&\quad+\frac{e^{\lte t}}{(\lte-\loe)\Lambda}\int_t^{T_\varepsilon} \xe^{p^+_\varepsilon}(s)e^{-\lte s}ds.
\end{split}
\end{equation}
This implies
\begin{equation}\label{terzultima}
\begin{array}{rl}
\xe(t) \leq &\displaystyle  |x_{\varepsilon,{}_-}|e^{\loe(t-T)}+\frac{e^{\loe t}}{(\lte-\loe)\Lambda}\int_T^t x^{p^+_\varepsilon}(s)e^{-\loe s}ds\\[2ex]
&  \displaystyle +\frac{e^{\lte t}}{(\lte-\loe)\Lambda}\int_t^{T_\varepsilon} x^{p^+_\varepsilon}(s)e^{-\lte s}ds.
\end{array}
\end{equation}
Moreover, since for $t>T$,  $\xe(t)<\delta$ and it is a decreasing  function,  we get
$$
e^{-\loe t}\xe(t)  \leq |x_{\varepsilon,{}_-}| e^{-\loe T}+\frac{\delta^{p^+_\varepsilon-1}}{(\lte-\loe)\Lambda}\int_T^te^{-\loe s}\xe(s)ds
 +\frac{\delta^{p^+_\varepsilon-1}e^{-\loe t}\xe(t)}{\lte(\lte-\loe)\Lambda}.
$$
We have chosen $\delta$ in \eqref{delta} such that 
$\beta:=\frac{\delta^{p^+_\varepsilon-1}}{\lte(\lte-\loe)\Lambda}<\frac12$, hence
$$
e^{-\loe t}\xe(t)\leq\frac{|x_{\varepsilon,{}_-}| e^{-\loe T}}{1-\beta}+\frac{\delta^{p^+_\varepsilon-1}}{(1-\beta)(\lte-\loe)\Lambda}\int_T^te^{-\loe s}\xe(s)ds.
$$
Using Gronwall's lemma, we get for  $\eta:=\frac{\delta^{p^+_\varepsilon-1}}{(\lte-\loe)\Lambda(1-\beta)}$ and  $C_T>\frac{|x_{\varepsilon,{}_-}|e^{-\loe T}}{1-\beta}$ 
\begin{equation}\label{penultima}
\xe(t)\leq C_T e^{(\loe+\eta)t}.
\end{equation}
Using \eqref{penultima} in \eqref{terzultima}, we get
\begin{equation*}
e^{-\loe t}\xe(t)\leq |x_{\varepsilon,{}_-}| e^{-\loe T}+\frac{C_T^{p^+_\varepsilon}}{(\lte-\loe)\Lambda}\int_T^t e^{(p^+_\varepsilon(\loe +\eta)-\loe)s}ds+\beta e^{-\loe t}\xe(t).
\end{equation*}
Finally, since $\delta$ has been chosen in order that $(\loe+\eta)p^+_\varepsilon<\loe$,   the integral on the right hand side is bounded independently of $\varepsilon$ small, and  we  then can choose  $C$ large enough but  independent of $\varepsilon$ such that
\begin{equation}\label{stima}
e^{-\loe t}\xe(t)\leq C.
\end{equation}
In order to obtain a similar estimate for $\xe^\prime$, we differentiate   equation  \eqref{xx2}: 
\begin{equation*}
\begin{split}
\xe^{\prime}(t)&=\loe\left[x_{\varepsilon,{}_-} e^{\loe(t-T)}+\frac{e^{\loe t}}{(\lte-\loe)\Lambda}\int_T^t x^{p_\varepsilon}(s)e^{-\loe s}ds\right]\\
&\quad+\lte\bigg[-x_{\varepsilon,{}_-} e^{\loe(T_\varepsilon-T)}e^{\lte(t-T_\varepsilon)} -\frac{e^{\loe T_\varepsilon}e^{\lte(t-T_\varepsilon)}}{(\lte-\loe)\Lambda}\int_T^{T_\varepsilon} x^{p_\varepsilon}(s)e^{-\loe s}ds\\
&\quad+\frac{e^{\lte t}}{(\lte-\loe)\Lambda}\int_t^{T_\varepsilon} x^{p_\varepsilon}(s)e^{-\lte s}ds\bigg]
\end{split}
\end{equation*}
Using \eqref{stima}, we then obtain that for some $C^\prime>0$ one has
\begin{equation}\label{stima'}
e^{-\loe t}|\xe^{\prime}(t)|\leq C^\prime.
\end{equation}
Recalling that $\lte-\loe=\np-2$, estimates \eqref{stima} and \eqref{stima'} written in terms of $\tilde{u}_\varepsilon^+$  give \eqref{phsp}.\\
 Next, for $r^+_\varepsilon\to +\infty$ as in the statement, let us set $t_\varepsilon=\log r^+_\varepsilon$. Then, $\te$ satisfies
$$
t_\varepsilon\to +\infty\, ,\quad \te \leq T_\varepsilon\,,\qquad T_\varepsilon-\te \to L_+=\log l_+\in [0,+\infty]\, .
$$
Evaluating $\xe(\te)$ by means of \eqref{xx2}, we obtain
\begin{equation}\label{xete}
\begin{array}{rl}
e^{-\loe \te}\xe (\te) = & \displaystyle x_{\varepsilon,-} e^{-\loe T}+\frac{1}{\Lambda (\lte-\loe)}\int_T^{\te} \xe^{p^+_\varepsilon}(s) e^{-\loe s} ds\\[2ex]
& \displaystyle -x_{\varepsilon,-}e^{-\loe T} e^{(\loe -\lte)(T_\varepsilon -\te)}
-\frac{e^{(\loe-\lte)(T_\varepsilon-\te)}}{\Lambda (\lte-\loe)}\int_T^{T_\varepsilon} \xe^{p^+_\varepsilon}(s) e^{-\loe s} ds\\[2ex]
& \displaystyle +\frac{e^{-(\loe-\lte) \te}}{\Lambda (\lte-\loe)}\int_{\te}^{T_\varepsilon} \xe^{p^+_\varepsilon}(s) e^{-\lte s} ds.
\end{array}
\end{equation}
By estimate \eqref{stima}, we easily obtain
$$
e^{-(\loe-\lte) \te} \int_{\te}^{T_\varepsilon} \xe^{p^+_\varepsilon}(s) e^{-\lte s}ds \leq \frac{C^{p^+_\varepsilon} e^{(p^+_\varepsilon-1)\loe \te}}{\lte}\to 0\ \hbox{as } \varepsilon\to 0,
$$
as well as, by using also  point 2) of Step 1,
$$
\lim_{\varepsilon \to 0} \int_T^{\te} \xe^{p^+_\varepsilon}(s) e^{-\loe s} ds= \lim_{\varepsilon \to 0} \int_T^{T_\varepsilon} \xe^{p^+_\varepsilon}(s) e^{-\loe s} ds= \int_T^{+\infty} X^{p^+_\star}(s) e^{-\lambda_1^\star s} ds.
$$
Hence, letting $\varepsilon \to 0$ in \eqref{xete} and recalling \eqref{relc1}, we deduce
$$
\begin{array}{rl}
 \displaystyle\lim_{\varepsilon \to 0} e^{-\loe \te}\xe (\te)= & \displaystyle \left(X_{-}e^{-\lambda^\star_1T} 
+\frac{1}{\Lambda (\lambda^\star_2-\lambda^\star_1)}\int_T^{+\infty}X^{\psp}(s) e^{-\lambda^\star_1 s}ds\right) \left(1-l_+^{\lambda_1^\star-\lambda_2^\star}\right)\\[2ex]
= &  \displaystyle c_1^+ \left(1-l_+^{\lambda_1^\star-\lambda_2^\star}\right).
\end{array}
$$
The above limit written in terms of $\tilde{u}^+_\varepsilon$ is precisely \eqref{conv}.

\end{proof}

\begin{corollary}\label{estic}
There exists a positive constant $C_{_{+}}$ such that for any $\tilde r^+_0(\varepsilon)\leq r\leq (M^+_\varepsilon)^{\frac{p^+_\varepsilon-1}{2}}$ one has
\begin{equation}\label{eqc}
\tuep(r)\leq \frac{\tuep(\tilde r_0(\varepsilon))}{\left(1+C_{_{+}}\tuep(\tilde r^+_0(\varepsilon))^{\frac{2}{\np-2}}(r^2-\tilde r^+_0(\varepsilon)^2)\right)^{\frac{\np-2}{2}}}\,.
\end{equation}
A similar estimate applies to $\tuem$.
\end{corollary}
\begin{proof}
Using Corollary \ref{estib} and Proposition \ref{esti}, for $r\geq\tilde r^+_0(\varepsilon)$ we get that 
$$
\frac{(\tuep)'(r)}{(\tuep)^{\frac{\np}{\np-2}}(r)r}\leq -C$$
where $C$ is a positive constant independent of $\varepsilon$. Integrating the above inequality we get
$$
\int_{r_0(\varepsilon)}^r-\frac{(\tuep)'(s)}{(\tuep)^{\frac{\np}{\np-2}}(s)}\,ds=\frac{\np-2}{2}\left((\tuep)^{-\frac{2}{\np-2}}(r)-(\tuep)^{-\frac{2}{\np-2}}(\tilde r_0(\varepsilon))\right)\geq\frac{C}{2}(r^2-\tilde r_0(\varepsilon)^2). 
$$
Then the conclusion follows.
\end{proof}
\begin{remark}\label{remarkestic}
{\rm
Corollary \ref{estic} can be expressed in terms of $u^+_\varepsilon$ in the following way:
there exists a positive constant $C_{_+}$ such that for $r\in[r_0^+(\varepsilon),1]$
\begin{equation}\label{eqd}
u^+_\varepsilon(r)\leq \frac{u^+_\varepsilon(r_0(\varepsilon))}{\left(1+C_{_+}{u^+_\varepsilon(r_0(\varepsilon))}^{\frac{2}{\np-2}}(M^+_\varepsilon)^{p_\varepsilon-\frac{\np}{\np-2}}(r^2-r_0(\varepsilon)^2)\right)^{\frac{\np-2}{2}}}.
\end{equation}
Analogously for $u^-_\varepsilon$.}
\end{remark}

\begin{proposition}\label{propderivata}
There exist positive constants $K_{{}_\pm}$ such that 
\begin{equation}\label{stimaderivata}
\lim_{\varepsilon\to0}{(M^\pm_\varepsilon)}^{\frac{p^\pm_\varepsilon(\tilde N_{{_\pm}}-2)-\tilde N_{{_\pm}}}{2}}(u^\pm_\varepsilon)'(1)=-K_{_{\pm}}.
\end{equation}
\end{proposition}
\begin{proof}
The rescaled function $\tilde u^+_\varepsilon$ satisfies the equation
\begin{equation}
\left((\tilde u^+_\varepsilon)'r^{\np-1}\right)'=-\frac{(\tilde u^+_\varepsilon)^{p_\varepsilon^+}}{\Lambda}r^{\np-1}\qquad\text{for $r\in\left[\tilde r_0^+(\varepsilon),(M^+_\varepsilon)^{\frac{p^+_\varepsilon-1}{2}}\right]$}.
\end{equation}
Integrating 
\begin{equation*}
\begin{split}
(\tuep)'\left((M^+_\varepsilon)^{\frac{p^+_\varepsilon-1}{2}}\right)&(M^+_\varepsilon)^{\frac{(p^+_\varepsilon-1)(\np-1)}{2}}=\\&=(\tuep)'(\tilde r_0^+(\varepsilon))(\tilde r_0^+(\varepsilon))^{\np-1}-\frac1\Lambda\int_{\tilde r_0^+(\varepsilon)}^{(M^+_\varepsilon)^{\frac{p^+_\varepsilon-1}{2}}}(\tuep)^{p_\varepsilon^+}(r)r^{\np-1}dr.
\end{split}
\end{equation*}
By the definition of $\tuep$ we deduce 

\begin{equation}\label{equ1}
{(M^+_\varepsilon)}^{\frac{p^+_\varepsilon(\np -2)-\np}{2}}(u^+_\varepsilon)'(1)=(\tuep)'(\tilde r_0^+(\varepsilon))(\tilde r_0^+(\varepsilon))^{\np-1}-\frac1\Lambda\int_{\tilde r_0^+(\varepsilon)}^{(M^+_\varepsilon)^{\frac{p^+_\varepsilon-1}{2}}}(\tuep)^{p_\varepsilon^+}(r)r^{\np-1}dr.
\end{equation}
By means of Proposition \ref{propfundamental} the first term in the right hand side of \eqref{equ1} converges, as $\varepsilon\to0$, to 
$$
(\Up)'(R^+_0)(R_0^+)^{\np-1}=-\frac{(\Up(R^+_0))^{\psp}(R_0^+)^{\np}}{\Lambda (\np-1)}.
$$ 
As far as the integral in formula \eqref{equ1} is concerned note that in view of Proposition \ref{esti}  and the fact $\psp>\frac{\np}{\np-2}$, the integrand is bounded, independently of $\varepsilon$ small  enough, by an integrable function. Hence 
$$
\lim_{\varepsilon\to0}{(M^+_\varepsilon)}^{\frac{p_\varepsilon(\tilde n-2)-\tilde n}{2}}(u_\varepsilon^+)'(1)=
-\frac{(\Up(R^+_0))^{\psp}(R_0^+)^{\np}}{\Lambda}-\frac1\Lambda\int_{R^+_0}^{+\infty}(\Up)^{\psp}(r)r^{\np-1}dr.
$$
The proof in the case of $\tuem$ is completely analogous.

\end{proof}

\begin{proof}[Proof of Theorem \ref{teorema1}]
Statements {\it i)}  and {\it iii) } follow respectively from Propositions \ref{propMepsilon}  and \ref{propfundamental}. As far as {\it ii)} is concerned, by means of uniformly elliptic estimates, it is sufficient to prove the convergence in $L^\infty_{\rm loc}(\overline{B}\backslash\left\{0\right\})$. Without loss of generality we prove the uniform convergence in spherical annuli. Let $0<r_1 < 1$ and let $A_{r_1}=\left\{x\in \overline{B}\,:\,|x|\geq r_1   \right\}$. For $\varepsilon$ small enough we can assume, by Proposition \ref{propfundamental}-{\it ii)}, that $r_0(\varepsilon)^2\leq\frac{r_1^2}{2}$. Hence in view of the estimate \eqref{eqd}, we get 
$$
\left\|u^\pm_\varepsilon\right\|_{L^\infty(A_{r_1})}\leq
\left(\frac{C_{{}_\pm}r_1^2}{2}(M^\pm_\varepsilon)^{\pepm-\frac{\npm}{\npm-2}}\right)^{-\frac{\npm}{\npm-2}}
$$

As $\varepsilon\to0$, $\frac{p^\pm_\varepsilon(\tilde N_{{}_\pm}-2)-\tilde N_{{}_\pm}}{2}\to\frac{p^*_{\pm}(\tilde N_{{}_\pm}-2)-\tilde N_{{}_\pm}}{2}>0$ and $M^\pm_\varepsilon\to+\infty$, see Proposition \ref{propMepsilon}. Hence $\left\|u^\pm_\varepsilon\right\|_{L^\infty(A_{r_1})}\to0$, as required.

Let us finally prove {\it iv)}, which is a direct application of Proposition \ref{esti}. We observe that, by definition,
$$
(M_\varepsilon^\pm)^{\frac{p^\pm_\varepsilon (\tilde{N}_\pm-2)-\tilde{N}_\pm}{2}} u_\varepsilon^\pm (r) =
\frac{\left( (M_\varepsilon^\pm)^{\frac{p_\varepsilon^\pm-1}{2}}r\right)^{\tilde{N}_\pm-2} \tilde{u}^\pm_\varepsilon \left( (M_\varepsilon^\pm)^{\frac{p_\varepsilon^\pm-1}{2}}r\right)}{r^{\tilde{N}_\pm-2}}.
$$
By applying \eqref{conv}  with $r_\varepsilon^\pm= (M_\varepsilon^\pm)^{\frac{p_\varepsilon^\pm-1}{2}}r$, we obtain
$$
(M_\varepsilon^\pm)^{\frac{p^\pm_\varepsilon (\tilde{N}_\pm-2)-\tilde{N}_\pm}{2}} u_\varepsilon^\pm (r)\to \frac{c_1^\pm}{r^{\tilde{N}_\pm-2}} \left( 1 -r^{\tilde{N}_\pm-2}\right)
$$
as $\varepsilon \to 0$,  for any fixed $r\in (0,1]$. Hence, by monotonicity,   the convergence is in $C_{\rm loc}\left(\overline{B}\setminus \{0\}\right)$ and then, by uniformly elliptic estimates,   in $C^2_{\rm loc}\left(\overline{B}\setminus \{0\}\right)$.

\end{proof}

\section{Energy invariance}\label{Sec4}

As recalled in Section \ref{Sec2}, the critical exponents $\psp$ and $\psm$ are the only exponents for which there exist fast decaying radial solutions $U^{{}^\pm}$ of \eqref{eq3}-\eqref{eq3'} satisfying \eqref{Ustar}-\eqref{Ustar'}-\eqref{eq3''}. Since in what  follows there will be no difference in considering the operator $\Mp$ or $\Mm$ we will detail everything for $\Mp$ and indicate at the end the obvious changes.\\
By the invariance under rescaling of the equation in \eqref{eq3}, it is obvious that there are infinitely many radial fast decaying solutions, all given by
\begin{equation}\label{4eq1}
U^{{}^+}_\alpha(r)=\alpha\, \Up\left(\alpha^{\frac{\psp-1}{2}}r\right),\quad\alpha>0.
\end{equation}
Now we consider the set $X$ of all $C^2$ positive radial functions in $\RN$ which change convexity only once, i.e.
\begin{equation*}
\begin{split}
X=\Big\{u\in C^2(\RN):\;&\text{$u$ radial, $u>0$ and $\exists r_0\in(0,+\infty)$ such that $u''(r_0)=0$, }\\
&\;\text{$u''(r)<0$ for $r\in(0,r_0)$ and $u''(r)>0$ for $r>r_0$}\Big\}.
\end{split}
\end{equation*}
Note that all functions $U^{{}^+}_\alpha$ belong to $X$ as recalled in Section \ref{Sec2}. For any $u\in X$ we define the radial weight:
\begin{equation}\label{4eq2}
g_u^+(x)=\begin{cases}
r_0^{\gamma^\star} & \text{if $|x|\leq r_0=r_0(u)$}\\
|x|^{\gamma^\star} & \text{if $|x|> r_0=r_0(u)$}
\end{cases}
\end{equation}
where $\displaystyle \gamma^\star=2\,\frac{\psp+1}{\psp-1}-N$. \\
Next we define the weighted energy integral:
\begin{equation}\label{4eq3}
\E(u)=\int_{\RN}u^{\psp+1}(x)g_u^+(x)\,dx\qquad\forall u\in X.
\end{equation} 

\begin{proposition}\label{4prop1}
For any function $u\in X$, $\E(u)$ is invariant under the  rescaling  \eqref{4eq1}. Moreover,  for any solution $U^{{}^+}_\alpha$ of \eqref{eq3}, i.e. $p=\psp$, the corresponding energy \eqref{4eq3} is finite, so we have
$$
\E(U^{{}^+}_\alpha)=\E(\Up)=:\Sigma^+<+\infty\qquad\forall\alpha>0.
$$
\end{proposition}
\begin{proof}
Let $u\in X$ and define the usual rescaling
$$
u_\alpha(x)=\alpha u\left(\alpha^{\frac{\psp-1}{2}}x\right),\qquad\alpha>0.
$$
It is obvious that $u_\alpha\in X$ and, using radial coordinates, $u_\alpha''(r)=0$ for $r=r_{0,\alpha}=r_0\,\alpha^{-{\frac{\psp-1}{2}}}$. Hence
\begin{equation*}
\begin{split}
\E(u_\alpha)&=\int_{\RN}{u_\alpha}^{\psp+1}(x)g^+_{u_\alpha}(x)\,dx\\
&=\frac{1}{\alpha^{\frac{\psp-1}{2}\gamma^\star-(\psp+1)}}\int_{\RN}u^{\psp+1}(\alpha^{\frac{\psp-1}{2}}x)g^+_{u}(\alpha^{\frac{\psp-1}{2}}x)\,dx\\
&=\frac{1}{\alpha^{\frac{\psp-1}{2}(\gamma^\star+N)-(\psp+1)}}\int_{\RN}u^{\psp+1}(y)g^+_{u}(y)\,dy\\
&=\E(u)
\end{split}
\end{equation*}
since $\gamma^\star+N=2\,\frac{\psp+1}{\psp-1}$ by definition.\\
Now we consider the solution $\Up$ of \eqref{eq3}, satisfying \eqref{Ustar}. We have
\begin{equation*}
\begin{split}
\E(\Up)&=\int_{\RN}{(\Up)}^{\psp+1}(|x|)g^+_{\Up}(x)\,dx\\
&=\int_{B_{R^+_0}}{(\Up)}^{\psp+1}(|x|){(R^+_0)}^{\gamma^\star}\,dx+\int_{\RN\backslash B_{R^+_0}}{(\Up)}^{\psp+1}(|x|) |x|^{\gamma^\star}\,dx<+\infty
\end{split}
\end{equation*}
since the second integral is finite because $\psp>\frac{\np}{\np-2}$.
\end{proof}

\begin{remark}\label{4rm1}
\rm
In the case when $\lambda=\Lambda=1$ and hence the Pucci's operators reduce to the Laplacian, the critical exponent is $p^\star=\frac{N+2}{N-2}$ so that $\gamma^\star=0$. Hence the integral $\E(u)$ in \eqref{4eq3} is just the $L^\frac{2N}{N-2}$ norm of $u$ and its invariance by rescaling is well know. Note that for the  finite energy positive solutions $v$ of the equation
\begin{equation}\label{4eq4}
-\Delta u=u^{p^\star}\qquad\text{in $\RN$}
\end{equation} 
it holds:
$$
\int_{\RN}{|\nabla v|}^2\,dx=\int_{\RN}v^{p^\star+1}\,dx.
$$
Hence the associated energy
\begin{equation}\label{4eq5}
J(v)=\frac12\int_{\RN}{|\nabla v|}^2\,dx-\frac{1}{p^\star+1}\int_{\RN}v^{p^\star+1}\,dx
\end{equation}
reduces to 
$$
\left(\frac12-\frac{1}{p^\star+1}\right)\E(v)=\frac1NS^N
$$
where $\displaystyle S=\inf_{u\in H^1_0(\Omega)\backslash\left\{0\right\}}\frac{\left\|Du\right\|_{L^2(\Omega)}}{\left\|u\right\|_{L^{2^\star}(\Omega)}}$ is the best Sobolev constant for the corresponding Sobolev embedding.\\
Analogously one could define the energy functionals:
$$
J^{{}^+}(u)=\frac12\int_{\RN}\left[-\Mp(D^2u)u\right]g^+_u(x)\,dx-\frac{1}{\psp+1}\int_{\RN}u^{\psp+1}g^+_u(x)\,dx
$$
which, on the solution of \eqref{eq3} with $p=\psp$, is equal to $\left(\frac12-\frac{1}{p^\star+1}\right)\E(u)$.\\ This is a reason for calling $\E(u)$ the energy of a solution. \\
Finally some computations as in Proposition \ref{4prop1} show easily that  the integral
$$
\int_{\RN}\left[-\Mp(D^2u)u\right]g^+_u(x)\,dx
$$
is also invariant by rescaling, for any function $u\in X$.
\end{remark}
\medskip
Now we consider the unique radial positive solution $u_\varepsilon$ of problem \eqref{eq1p} with $p_\varepsilon=\psp-\varepsilon$, for $\varepsilon$ sufficiently small and $B$ the unit ball as in the previous sections.\\
In the case when the Pucci's operator reduce to the Laplacian (i.e. when $\Lambda=\lambda=1$) it is well known that $u_\varepsilon$ is the least-energy solution of the \lq\lq almost critical\rq\rq\ problem and for the functional $J$ defined in \eqref{4eq5} it holds:
\begin{equation}\label{4eq1'}
\begin{split}
J(u_\varepsilon)&=\frac12\int_B|\nabla u_\varepsilon|^2\,dx-\frac{1}{p_\varepsilon+1}\int_Bu^{p_\varepsilon+1}_\varepsilon\,dx\\
&=\left(\frac12-\frac{1}{p_\varepsilon+1}\right)\int_Bu^{p_\varepsilon+1}_\varepsilon\,dx\quad\stackrel{\varepsilon\to0}{\longrightarrow}\quad\frac1NS^N\,.
\end{split}
\end{equation}
In other words, by Remark \ref{4rm1}
\begin{equation}\label{4eq6}
\int_Bu^{p_\varepsilon+1}_\varepsilon\,dx\quad\stackrel{\varepsilon\to0}{\longrightarrow}\quad\int_{\RN}v^{p^\star+1}\,dx
\end{equation}
where $v$ is a positive solution of  \eqref{4eq4}.\\
Analogously we consider the weighted energies:
\begin{equation}\label{4eq7}
\Eeps(u_\varepsilon^+)=\int_{B}{(u_\varepsilon^+)}^{p^+_\varepsilon+1}g^+_{u^+_\varepsilon}\,dx
\end{equation}
with 
\begin{equation}\label{4eq8}
g_{u_\varepsilon^+}^+(x)=\begin{cases}
r_0^{\gamma_\varepsilon} & \text{if $|x|\leq r_0=r_0(u_\varepsilon^+)$}\\
|x|^{\gamma_\varepsilon} & \text{if $|x|> r_0=r_0(u_\varepsilon^+)$}
\end{cases}
\quad\text{and}\quad\gamma_\varepsilon=2\,\left(\frac{p_\varepsilon^++1}{p_\varepsilon^--1}\right)-N\,.
\end{equation}
Note that this radial weight is well defined on functions in the space
\begin{equation*}
\begin{split}
X_{B}=\Big\{u\in C^2(B):\;&\text{$u$ radial, $u>0$ and $\exists r_0\in(0,1)$ such that $u''(r_0)=0$, }\\
&\;\text{$u''(r)<0$ for $r\in(0,r_0)$ and $u''(r)>0$ for $r\in(r_0,1)$}\Big\}.
\end{split}
\end{equation*}

\medskip

We can now prove Theorem \ref{4th1}.

\begin{proof}[Proof of Theorem \ref{4th1}]
With the same notations as in the previous sections, let us consider the rescaled functions:
$$
\tuep(x)=\frac{1}{M_\varepsilon^+}u^+_\varepsilon\left(\frac{x}{{(M_\varepsilon^+)}^{\frac{p^+_\varepsilon-1}{2}}}\right)
$$
which converge to $\Up$ in $C^2_{\rm{loc}}(\RN)$. Denoting by $\tilde r_0^+(\varepsilon)=r_0^+(\varepsilon){(M_\varepsilon^+)}^{\frac{p^+_\varepsilon-1}{2}}$ the radius where $(\tuep)''$ vanishes, we  have for the weight $g_{\tuep}^+$
\begin{equation*}
\begin{split}
g_{\tuep}^+(x)&=\begin{cases}
{[\tilde r_0^+(\varepsilon)]}^{\gamma_\varepsilon} & \text{if $|x|\leq \tilde r_0^+(\varepsilon)$}\\
|x|^{\gamma_\varepsilon} & \text{if $|x|> \tilde r_0^+(\varepsilon)$}
\end{cases}\\
&={(M_\varepsilon^+)}^{\frac{p^+_\varepsilon-1}{2}\,\gamma_\varepsilon}g_{u_\varepsilon^+}^+\left(\frac{x}{{(M_\varepsilon^+)}^{\frac{p^+_\varepsilon-1}{2}}}\right).
\end{split}
\end{equation*}
Therefore, setting $\varrho_\varepsilon={(M_\varepsilon^+)^\frac{p^+_\varepsilon-1}{2}}$, 
\begin{equation*}
\begin{split}
\Eeps(\tuep)&=\int_{B_{\varrho_\varepsilon}}{(\tuep)}^{p^+_\varepsilon+1}(x)g_{\tuep}^+(x)\,dx\\
&={(M_\varepsilon^+)}^{\frac{p^+_\varepsilon-1}{2}\,\gamma_\varepsilon-(p^+_\varepsilon+1)}\int_{B_{\varrho_\varepsilon}}{(u^+_\varepsilon)}^{p^+_\varepsilon+1}\left(\frac{x}{\varrho_\varepsilon}\right)g_{u_\varepsilon^+}^+\left(\frac{x}{\varrho_\varepsilon}\right)\,dx\\
&={(M_\varepsilon^+)}^{\frac{p^+_\varepsilon-1}{2}\,(\gamma_\varepsilon+N)-(p^+_\varepsilon+1)}\int_B{(u^+_\varepsilon)}^{p^+_\varepsilon+1}\left(y\right)g_{u_\varepsilon^+}^+\left(y\right)\,dy\,.
\end{split}
\end{equation*}
Applying the dominated convergence theorem since $\tuep\to\Up$ locally uniformly by Proposition \ref{propfundamental} and Corollary \ref{estic} holds, by definition of $\gamma_\varepsilon$ in \eqref{4eq8} we have
$$
\Eeps(u^+_\varepsilon)=\Eeps(\tuep)\to\E(\Up) \quad\text{as $\varepsilon\to0$.}
$$
\end{proof}

Note that \eqref{4eq9} represents the counterpart of \eqref{4eq1'} for the fully nonlinear problem \eqref{eq1p} with $\Sigma^+$ playing the same role as the constant $\frac{1}{N}S^N$.

All the results of this section can be stated for $\Um$ and $u^-_\varepsilon$ with obvious changes, in particular considering the weights
$$
\gamma^\star=2\,\frac{\psm+1}{\psm-1}-N\qquad\text{and}\qquad\gamma_\varepsilon=2\,\frac{p^-_\varepsilon+1}{p^-_\varepsilon-1}-N
$$
instead of those defined in \eqref{4eq2} and \eqref{4eq8} in the corresponding energies $\Eems$ and $\Em$.

\begin{remark}
\rm
Let us make some final comments which may clarify the role of the weighted energies defined in \eqref{4eq3} by means of \eqref{4eq2}.\\ The difficulty in dealing with equations involving the Pucci's operators is that these operators can be written in different ways according to the change of convexity of the radial function in the ball to which they are applied. More precisely, looking at \eqref{eqU+}-\eqref{eqU-}, for a function $u$ belonging to the set $X$ we have
$$
\Mp(D^2u)=\lambda\Delta u\qquad\text{in $B_{r_0}$, $r_0=r_0(u)$}
$$
while
$$
\Mp(D^2u)=\Lambda\left[u''+(\np-1)\frac{u'}{r}\right]\qquad\text{for  $r>r_0$.}
$$
Let us then consider only the second operator in the whole $[0,+\infty)$, i.e. we define 
$$
\D(u)=u''+(\np-1)\frac{u'}{r}
$$
for any radial function $u$ in $\RN$ (in other words we neglect the laplacian part of $\Mp$). Then it is not difficult to see that the natural related critical exponent for this operator is
$$
p^\star_{\D}=\frac{\np+2}{\np-2}
$$
and several results holding for the Laplacian could be transferred to $\D$. In particular the energy invariance, that we have described before, holds taking the weighted energy
$$
E_{\D}(u)=\int_{\RN}u^{p^\star_{\D}+1}(x)|x|^{\np-N}\,dx\,.
$$ 
In the case of the Pucci's operator $\Mp$ the exponent $\gamma^\star_+$ in the definition of \eqref{4eq2} satisfies $$\np-N\leq\gamma^\star\leq0.$$
Therefore the weighted energy $\E(u)$ we have considered in \eqref{4eq3} represent a kind of \lq\lq interpolation\rq\rq between the energy defined for  the Laplacian ($\gamma^\star=0$) and the \lq\lq energy\rq\rq related to the operator $\D$ ($\gamma^\star=\np-N$). It shows how the energy should be corrected passing from the Laplacian to the operator $\D$.

\medskip
Similar comments can be made for the Pucci's operator $\Mm$ and the related weighted energies.

\end{remark}

\end{document}